\documentclass[11pt,leqno]{article}
\usepackage{calder}

\title{An Elementary Proof of a Minimax Theorem}
\author{Jeff Calder}
\affil{School of Mathematics \\ University of Minnesota\thanks{{\bf Email} \textit{jwcalder@umn.edu}, {\bf Funding:} This work was supported by NSF grant 2436333, and an Albert and Dorothy Marden Professorship.}}

\begin{document} 
\maketitle

\begin{abstract}
Here, we give a self-contained and elementary proof of a minimax theorem due to Fan \cite{fan1953minimax} in a simplified setting that can be taught in an advanced undergraduate course. Our proof follows \cite{nikaido1959method} with some simplifications. 
\end{abstract}

\section{Introduction}

Let $X\subset \R^d$, $Y\subset \R^k$, and $f:X\times Y \to \R$.
We always assume $X$ and $Y$ are nonempty. We are concerned here with \emph{minimax theorems}, which give conditions under which it holds that
\begin{equation}\label{eq:minimax}
\sup_{\x\in X}\inf_{\y\in Y} f(\x,\y) = \inf_{\y\in Y}\sup_{\x\in X} f(\x,\y).
\end{equation}
The purpose of this note is to provide an elementary proof of a minimax theorem that is accessible to undergraduate students.  Minimax theorems date back to von Neumann \cite{v1928theorie} in 1928, who established the first minimax theorem in the special case that $f(\x,\y) = \x^TA\y$ and $X$ and $Y$ are the spaces of probability vectors in $\R^d$ and $\R^k$, respectively. This initial setting arose in two player zero sum games.  Since then, many extensions to more general functions $f(\x,\y)$ have been given (see, e.g., \cite{simons1995minimax} and references therein), and minimax theorems have found powerful applications in primal dual optimization. 

One way to think about the difference between the two sides of \eqref{eq:minimax} is in terms of two player games, where the first player is choosing $\x\in X$ to maximize $f(\x,\y)$, while the second player is choosing $\y\in Y$ to minimize $f(\x,\y)$. When the supremum is on the outside, like on the left hand side of \eqref{eq:minimax}, this amounts to the second player getting to observe the first player's choice $\x$ before making their choice $\y$. This confers an advantage to the second player. Vice versa, when the infimum is on the outside, like on the right hand side of \eqref{eq:minimax}, the first player gets to observe the second player's choice $\y$ before making their choice of $\x$, which confers an advantage to the first player. Thus, we expect to always see an inequality, which is verified in the following result.
\begin{lemma}\label{lem:minimaxless}
Let $X\subset \R^d$, $Y\subset \R^k$, and $f:X\times Y \to \R$.
Then 
\begin{equation}\label{eq:minimaxless}
\sup_{\x\in X} \inf_{\y\in Y} f(\x,\y) \leq \inf_{\y\in Y} \sup_{\x\in X} f(\x,\y).
\end{equation}
\end{lemma}
\begin{proof}
For any $\x_0\in X$ we have
\[\inf_{\y\in Y} f(\x_0,\y) \leq \inf_{\y\in Y} \sup_{\x\in X} f(\x,\y).\]
Taking the supremum over $\x_0\in X$ gives \eqref{eq:minimaxless}.
\end{proof}
By Lemma \ref{lem:minimaxless}, to prove a minimax theorem, it is sufficient to show the opposite inequality, namely that
\begin{equation}\label{eq:opp_ineq}
\inf_{\y\in Y} \sup_{\x\in X} f(\x,\y) \leq \sup_{\x\in X} \inf_{\y\in Y} f(\x,\y).
\end{equation}
However, this direction is much more difficult to prove, and does not always hold. 
\begin{example}\label{Ex:difference}
Let $X = \{-1,1\}\subset \R$, $Y = [-1,1]\subset \R$, and $f(x,y) = xy$. Then 
\[\max_{x\in X}\min_{y\in Y}f(x,y) = \max_{x\in X} (-1) = -1,\]
while 
\[\min_{y\in Y}\max_{x\in X}f(x,y) = \min_{y\in Y} |y| = 0.\]
\end{example}
As we shall see below, convexity of the sets $X$ and $Y$ play a key role in minimax theorems. Note that $X=\{-1,1\}$ is not a convex set.

\section{Minimax and Saddle Points}

We now introduce some notation. We say that $\xstar\in X$ is an \emph{outer max} of $f$ if the supremum on the left hand side of \eqref{eq:minimax} is attained at $\xstar$. Likewise, we say $\ystar\in Y$ is an \emph{outer min} of $f$ if the infimum on the right hand side of \eqref{eq:minimax} is attained at $\ystar$.  We also define the notion of a \emph{saddle point} of $f$.
\begin{definition}\label{def:saddle_point}
A pair $(\xstar,\ystar)\in X\times Y$ is called a \emph{saddle point} of $f$ if 
\begin{equation}\label{eq:saddle}
f(\x,\ystar) \leq f(\xstar,\ystar) \leq f(\xstar,\y) \ \ \text{for all} \ \ (\x,\y)\in X\times Y.
\end{equation}
\end{definition}
A saddle point $(\xstar,\ystar)\in X\times Y$ of $f$ clearly satisfies
\begin{equation}\label{eq:saddle2}
\max_{\x\in X}f(\x,\ystar) = f(\xstar,\ystar) = \min_{\y \in Y} f(\xstar,\y).
\end{equation}
In fact, \eqref{eq:saddle} is equivalent to \eqref{eq:saddle2}. In other words, $(\xstar,\ystar)$ is a saddle point if $\x\mapsto f(\x,\ystar)$ is maximized at $\xstar$, and $\y\mapsto f(\xstar,\y)$ is minimized at $\ystar$. 

When the supremum and infimum are attained, minimax theorems are equivalent to the problem of finding a \emph{saddle point}.
\begin{proposition}\label{prop:saddle}
Let $f:X\times Y \to \R$. If $f$ has a saddle point $(\xstar,\ystar)\in X\times Y$, then \eqref{eq:minimax} holds and $\xstar$ is an outer max of $f$, while $\ystar$ is an outer min of $f$. Conversely, if \eqref{eq:minimax} holds, $\xstar\in X$ is an outer max of $f$, and $\ystar\in Y$ is an outer min of $f$, then $(\xstar,\ystar)$ is a saddle point of $f$. 
\end{proposition}
\begin{proof}
If $(\xstar,\ystar)\in X\times Y$ is a saddle point, so \eqref{eq:saddle} holds, then we have
\[\inf_{\y\in Y}\sup_{\x\in X}f(\x,\y) \leq \max_{\x\in X}f(\x,\ystar) = f(\xstar,\ystar),\]
and
\[\sup_{\x\in X}\inf_{\y\in Y}f(\x,\y) \geq \min_{\y\in Y}f(\xstar,\y) = f(\xstar,\ystar)\]
which yields \eqref{eq:opp_ineq}. By Lemma \ref{lem:minimaxless} we have that \eqref{eq:minimax} holds, and so 
\[\inf_{\y\in Y}\sup_{\x\in X}f(\x,\y) = f(\xstar,\ystar) = \sup_{\x\in X}\inf_{\y\in Y}f(\x,\y).\]
It follows that $\xstar$ is an outer max of $f$, while $\ystar$ is an outer min. 

Conversely, suppose that \eqref{eq:minimax} holds, and let $\xstar$ be an outer max of $f$ and $\ystar$ an outer min. This means that 
\[\sup_{\x\in X}\inf_{\y \in Y} f(\x,\y) = \inf_{\y \in Y}f(\xstar,\y) \leq f(\xstar,\ystar),\]
and
\[\inf_{\y \in Y}\sup_{\x\in X} f(\x,\y) = \sup_{\x\in X}f(\x,\ystar) \geq f(\xstar,\ystar).\]
Since \eqref{eq:minimax} holds we have
\[\sup_{\x\in X} f(\x,\ystar) = f(\xstar,\ystar) = \inf_{\y \in Y} f(\xstar,\y),\]
which implies that $(\xstar,\ystar)$ is a saddle point of $f$.
\end{proof}

\section{A Minimax Theorem}

We now give our main result. The proof follows Nikaid\^o's argument \cite{nikaido1959method}, with some modifications.
\begin{theorem}\label{thm:minimax}
Let $X\subset \R^d$ and $Y\subset \R^k$ be convex and compact sets with nonempty interiors. Let $f:X\times Y \to \R$ be continuous, and assume $\x\mapsto f(\x,\y)$ is concave for all $\y\in Y$ and $\y\mapsto f(\x,\y)$ is convex for all $\x\in X$.
Then $f$ has a saddle point $(\xstar,\ystar)\in X\times Y$.
\end{theorem}
\begin{proof}
Let us define 
\[\Phi(\x,\y) = \int_Y \int_X (f(\u,\y) - f(\x,\v))_+^2 \d\u \d\v,\]
where $a_+ = \max(a,0)$.\footnote{Some care should be taken in defining the integral. Since $X$ and $Y$ are convex, they are measurable sets. Since they are compact and convex with \emph{nonempty} interiors, each is the closure of an open and bounded set, which has positive Lebesgue measure. Then we can write $\int_X f(\x)\d\x = \int_{\R^d}f(\x)\one_X(\x)\d\x$ for any continuous $f:X\to \R$, where $\one_X$ is the indicator function of the set $X$. In particular, if $f:X\to \R$ is continuous and \emph{nonnegative} and $\int_X f\d\x=0$, then $f\equiv 0$ on $X$.} Since $f$ is continuous on the compact set $X\times Y$, so is $\Phi$, and hence $\Phi$ attains its minimum at a point $(\xstar,\ystar)\in X\times Y$.
The proof will be completed by showing that $\Phi(\xstar,\ystar)=0$, from which we obtain that $(\xstar,\ystar)$ is a saddle point.

To see that $\Phi(\xstar,\ystar)=0$, we will take a variation of $\Phi$ that remains inside our convex domain $X\times Y$.
For any $0 < t < 1$ and $(\x,\y)\in X\times Y$ we use convexity/concavity of $f$ to obtain 
\begin{align*}
\Phi(\xstar,\ystar) &\leq \Phi(\xstar + t(\x-\xstar),\ystar + t(\y-\ystar))\\
&=\int_Y \int_X (f(\u,(1-t)\ystar + t\y) - f((1-t)\xstar + t\x,\v))_+^2 \d\u \d\v\\
&\leq \int_Y \int_X [\,(1-t)f(\u,\ystar) + tf(\u,\y) - (1-t)f(\xstar,\v) - tf(\x,\v)\,]_+^2 \d\u \d\v.
\end{align*}
Writing $g(\u,\v) = f(\u,\ystar)-f(\xstar,\v)$ and $h(\u,\v;\x,\y) =f(\u,\y)- f(\x,\v)$ for notational convenience, we can rearrange the above to obtain  
\begin{equation}\label{eq:key}
\int_Y\int_X [\, g(\u,\v) + t(h(\u,\v;\x,\y) - g(\u,\v))\,]_+^2 - g(\u,\v)_+^2\d\u \d\v \geq 0.
\end{equation}
We now divide by $t>0$, take the limit as $t\to 0^+$, and use that
\begin{align*}
\lim_{t\to 0^+}\frac{1}{t}\left[(a + tb)^2_+ - a^2_+\right]  = 2a_+ b
\end{align*}
for all $a,b\in\R$.
This yields
\[\int_Y\int_X g(\u,\v)_+ (h(\u,\v;\x,\y) - g(\u,\v)) \d\u \d\v \geq 0,\]
which can be rearranged to read
\begin{equation}\label{eq:variation}
\Phi(\xstar,\ystar) \leq \int_Y\int_X g(\u,\v)_+h(\u,\v;\x,\y)  \d\u \d\v.
\end{equation}
Multiply by $g(\x,\y)_+$ on both sides and integrate over $(\x,\y)\in X\times Y$ to obtain 
\begin{align*}
&\Phi(\xstar,\ystar)\int_Y\int_X g(\x,\y)_+\d\x \d\y \\
&\hspace{1in}\leq \int_Y\int_X\int_Y\int_X g(\u,\v)_+g(\x,\y)_+h(\u,\v;\x,\y)  \d\u \d\v\d\x \d\y=0
\end{align*}
since $h$ is skew-symmetric (i.e., $h(\u,\v;\x,\y) = -h(\x,\y;\u,\v)$).
Therefore $\Phi(\xstar,\ystar)=0$.
\end{proof}

In many applications of minimax theorems, one of the sets $X$ or $Y$ is unbounded, and in particular, not compact, though $f$ has more structure. The minimax theorem of Fan \cite{fan1953minimax} applies in this setting, though the proof is rather involved. Here, we take the simpler route of extending Theorem \ref{thm:minimax} to the unbounded case when $f$ has quadratic dependence on $\y$ and the mixed terms are linear. In particular, we assume now that $f$ has the form 
\begin{equation}\label{eq:quadf}
f(\x,\y) = \frac{1}{2}\y^TS\y - \y^TA\x - g(\x),
\end{equation}
for appropriate conditions on $S$, $A$, and $g$. In particular, we prove the following result.
\begin{theorem}\label{thm:minimax_unbounded}
Let $X\subset \R^d$ be convex and compact with nonempty interior and assume $\o\in X$. Let $f:X\times \R^k \to \R$ be given by \eqref{eq:quadf}, where $g:X\to \R$ is continuous and convex, $S=S^T$ is a $k\times k$ symmetric and positive semidefinite matrix, and $A$ is any $k\times d$ matrix.
Then it holds that
\begin{equation}\label{eq:saddle_quad}
\inf_{\y\in \R^k}\max_{\x\in X} f(\x,\y) = \max_{\x\in X}\inf_{\y\in \R^k}f(\x,\y).
\end{equation}
\end{theorem}
\begin{proof}
For any fixed $\x\in X$, $\y\mapsto f(\x,\y)$ is a quadratic function which has a global minimizer if and only if $A\x\in \img S$.  This holds, for example, when $\x=\o$. In the case $A\x\in \img S$, every minimizer $\y$ satisfies $S\y = A\x$ and there is a unique minimizer $\y\in \img S$, which satisfies  
\[\sigma_{min}(S) \|\y\| \leq \|S\y\| = \|A\x\| \leq \|A\| \|\x\| \leq C \|A\| \ \ \text{where}  \ \ C=\max_{\x\in X}\|\x\|,\]
where $\sigma_{min}(S)>0$ is the smallest (positive) singular value of $S$.
Thus, when $A\x\in \img S$, $\y\mapsto f(\x,\y)$ admits its minimum value over $\R^k$ on the closed ball $B_R$ of radius $R = C\|A\|/\sigma_{min}(S)$.  When $A\x\not\in \img S$ we have $\inf_{\y\in \R^k}f(\x,\y) = -\infty$, so the maximum over $\x\in X$ will avoid such choices. It follows that the outer max of $f$ is attained and we have 
\[\max_{\x\in X}\inf_{\y\in \R^k}f(\x,\y) = \max_{\x\in X}\min_{\y\in B_R}f(\x,\y).\]

We can now apply Theorem \ref{thm:minimax} to find that 
\[\max_{\x\in X}\inf_{\y\in \R^k}f(\x,\y) = \min_{\y\in B_R}\max_{\x\in X}f(\x,\y) \geq \inf_{\y\in \R^k}\max_{\x\in X}f(\x,\y),\]
which when combined with Lemma \ref{lem:minimaxless} completes the proof.
\end{proof}

\begin{remark}\label{rem:saddlequad}
By Theorem \ref{thm:minimax_unbounded}, $f$ has an outer max $\xstar\in X$, and $A\xstar \in \img S$. If $f$ has an outer min $\ystar\in \R^k$ as well, then by Proposition \ref{prop:saddle}, the pair $(\xstar,\ystar)\in X\times \R^k$ is a saddle point, that is 
\begin{equation}\label{eq:saddle3}
\min_{\y\in \R^k}\max_{\x\in X} f(\x,\y) = \max_{\x\in X}f(\x,\ystar) = f(\xstar,\ystar) = \min_{\y \in \R^k} f(\xstar,\y)=\max_{\x\in X}\inf_{\y\in \R^k}f(\x,\y).
\end{equation}
We also note that we can clearly restrict the outer max to the set 
\[X_0 = \{\x\in X \st A\x\in \img S\},\]
in which case the inner infimum is a true minimum and we can write
\begin{equation}\label{eq:saddle4}
\min_{\y\in \R^k}\max_{\x\in X} f(\x,\y) = \max_{\x\in X}f(\x,\ystar) = f(\xstar,\ystar) = \min_{\y \in \R^k} f(\xstar,\y)=\max_{\x\in X_0}\min_{\y\in \R^k}f(\x,\y).
\end{equation}
\end{remark}

\subsection*{Acknowledgements}

We thank Markus Schweighofer for discussions that led to this note.

\bibliography{main}
\bibliographystyle{abbrv}

\end{document}